\documentclass{amsart}
\usepackage{amssymb,latexsym}
\usepackage{amscd,amsthm}

\usepackage[all]{xy}

\newtheorem{theorem}{Theorem}[section]
\newtheorem{lemma}[theorem]{Lemma}
\newtheorem{proposition}[theorem]{Proposition}
\newtheorem{corollary}[theorem]{Corollary}

\theoremstyle{definition}
\newtheorem{definition}[theorem]{Definition}

\newtheorem*{remark}{Remark}

\DeclareMathOperator{\Ext}{Ext}

\DeclareMathOperator{\im}{Im}

\newcommand{\cat}[1]{\mathcal{#1}}           

\newcommand{\class}[1]{\mathcal{#1}}   

\newcommand{\Z}{\mathbb{Z}}

\begin{document}

\title{The homotopy category of $N$-complexes is a homotopy category}

\date{\today}

\author{James Gillespie}

\begin{abstract}
We show that the category of $N$-complexes has a Str\o m model structure, meaning the weak equivalences are the chain homotopy equivalences. This generalizes the analogous result for the category of chain complexes ($N = 2$). The trivial objects in the model structure are the contractible $N$-complexes which we necessarily study and derive several results.
\end{abstract}

\maketitle


\section{Introduction}\label{section-introduction}

Let $R$ be a ring and $N \geq 2$. By an $N$-complex $X$ we mean a sequence of $R$-modules and $R$-linear maps
$$\cdots \xrightarrow{d_{n+2}} X_{n+1} \xrightarrow{d_{n+1}} X_n \xrightarrow{d_n} X_{n-1} \xrightarrow{d_{n-1}} \cdots$$ satisfying $d^N = 0$. That is, composing any $N$-consecutive maps gives 0. So a 2-complex is a chain complex in the usual sense. $N$-complexes seem to have first appeared in the paper~\cite{kapranov}. Since then many papers have appeared on the subject, many of them studying their interesting homology (recently called ``amplitude homology''), and pointing to their relevance in theoretical physics. See for example~\cite{kassel-wambst}, \cite{dubois-violette}, \cite{tikaradze}, \cite{estrada}, \cite{amplitude cohomology}, \cite{henneaux}, and \cite{gill-hovey-generalized derived cats}. There are many other papers written on the subject, most notably those of M. Dubois-Violette and coauthors.

Recall that Quillen's notion of a model structure on a category provides a context for a homotopy theory in that category. Quillen's original model structure on the category of topological spaces has as weak equivalences the weak homotopy equivalences~\cite{quillen}. This is the canonical example of a model structure and its associated homotopy category is equivalent to the usual homotopy category of CW-complexes. On the other hand, Arne Str\o m proved in~\cite{strom} that the category of all topological spaces has a model category structure where the weak equivalences are the (strong) homotopy equivalences. The homotopy category associated to this model structure recovers the more naive homotopy category in which morphisms between spaces are homotopy classes of continuous maps.

There is an analogous situation for the category of chain complexes of $R$-modules. In Chapter~2.3 of~\cite{hovey-model-categories}, Hovey describes a projective model structure on chain complexes having as weak equivalences the homology isomorphisms. The associated homotopy category is the unbounded derived category $\class{D}(R)$. (Quillen originally did this for bounded below chain complexes.) But there is a Str\o m-type model structure on chain complexes as well which has as weak equivalences the chain homotopy equivalences. In analogy with topological spaces, the resulting homotopy category is the naive homotopy category where maps are homotopy classes of chain maps. This was the result proved in the paper~\cite{golasinski and gromadzki}.

And so the same should be true for the category of $N$-complexes. In~\cite{gill-hovey-generalized derived cats}, the authors constructed a Quillen model structure on the category of $N$-complexes which generalizes the usual projective model structure on chain complexes constructed in chapter~2.3 of~\cite{hovey-model-categories}. This model structure on $N$-complexes can be viewed as a model for amplitude homology theory since the weak equivalences are the amplitude homology isomorphisms. The main result of the current paper is the existence of a Str\o m type model structure on $N$-complexes. This statement appears in Theorem~\ref{them-the homotopy category of $N$-complexes is a homotopy category}.

Our techniques are entirely different than those in~\cite{golasinski and gromadzki}. We use Hovey's method of cotorsion pairs to construct the model structure. This method was written in the language of exact categories in~\cite{gillespie-exact model structures}. We will see that the model structure is ``Frobenius'' in the sense that it exists on an exact category and every object is both cofibrant and fibrant.

The paper should be quite accessible to anyone with just a bit of familiarity with chain complexes and either model categories or cotorsion pairs. In Section~\ref{section-preliminaries} we give a summary of any background information needed on $N$-complexes and cotorsion pairs/model categories. In Section~\ref{section-contractible $N$-complexes} we make a brief study of contractible $N$-complexes, which are the trivial objects in the model structure. In particular, we characterize contractible complexes as direct sums of $N$-disks in Theorem~\ref{them-contractible N-complexes are direct sums of $N$-disks} and as the projective and injective objects in an exact category in Proposition~\ref{prop-contractible complexes are projective and injective}. We also prove that two chain maps are homotopic if and only if their difference factors through a contractible $N$-complex (Corollary~\ref{cor-factoring through a contractible}). The main result is proved in Section~\ref{section-main theorem} as Theorem~\ref{them-the homotopy category of $N$-complexes is a homotopy category}.

\section{Preliminaries: $N$-complexes and Hovey pairs}\label{section-preliminaries}

In this section we review the central concepts that are related in this paper: $N$-complexes and model structures. We provide references to the literature for more complete explanations.

\subsection{The category of $N$-complexes} We will mostly follow the original notation and definitions of~\cite{kapranov} and~\cite{kassel-wambst} when working with $N$-complexes.

Throughout this paper $R$ denotes a ring with unity and $N \geq 2$ is an integer. One should think of an $N$-complexes as a generalized chain complex. Precisely, an \emph{$N$-complex} is a sequence of $R$-modules and maps $$\cdots \xrightarrow{d_{n+2}} X_{n+1} \xrightarrow{d_{n+1}} X_n \xrightarrow{d_n} X_{n-1} \xrightarrow{d_{n-1}} \cdots$$ satisfying $d^N = 0$. That is, composing any $N$-consecutive maps gives 0. So a 2-complex is chain complex in the usual sense. A \emph{chain map} or simply \emph{map} $f : X \xrightarrow{} Y$ of $N$-complexes is a collection of maps $f_n : X_n \xrightarrow{} Y_n$ making all the rectangles commute. In this way we get a category of $N$-complexes, denoted $N\text{-Ch}(R)$, whose objects are $N$-complexes and whose morphisms are chain maps. This is an abelian category with all limits and colimits taken degreewise.

Given an $R$-module $M$, we define an $N$-complex $D_n(M)$ by letting it equal $M$ in
degrees $n, n-1 , n-2, \cdots , n - (N-1)$ , all joined by identity
maps, and 0 in every other degree. We will call it the \emph{disk on $M$ of degree $n$}. So when $N = 2$, we get that $D_n(M)$ is the usual disk on $M$ used in algebraic topology.

Next, for an $N$-complex $X$ note that there are $N-1$ choices for
homology. Indeed for $t = 1 , 2 \cdots , N$ we define ${}_tZ_n(X) =
\ker{(d_{n-(t-1)} \cdots d_{n-1}d_n)}$. In particular, we have
${}_1Z_n(X) = \ker{d_n}$ and ${}_NZ_n(X) = X_n$. Next, for $t = 1 , 2
\cdots , N$ we define ${}_tB_n(X) = \im{(d_{n+1} d_{n+2} \cdots
d_{n+t})}$. In particular, ${}_1B_n(X) = \im{d_{n+1}}$ and ${}_NB_n(X) =
0$. Finally, we define ${}_tH_n(X) = {}_tZ_n(X)/{}_{N-t}B_n(X)$ for $t =
1,2,\cdots,N-1$. Following~\cite{amplitude cohomology} we call these modules the amplitude homology modules of $X$.

\begin{definition}
Let $X$ be an $N$-complex. We call ${}_tH_n(X)$ the \emph{amplitude $t$ homology module of degree $n$} (or the \emph{$n^{\text{th}}$ amplitude $t$ homology module of $X$}). We say $X$ is \emph{$N$-exact}, or
just \emph{exact}, if ${}_tH_n(X) = 0$ for each $n$ and all $t =
1,2,\cdots,N-1$.
\end{definition}

The facts in the following proposition are fundamental.

\begin{proposition}\label{prop-exactness of $N$-complexes} We have the following properties on exactness of $N$-complexes.
\begin{enumerate}
\item An $N$-complex $X$ is exact if and only if for any fixed amplitude $t$ we have ${}_tH_n(X) = 0$ for each $n$.

\item Suppose $0 \xrightarrow{} X \xrightarrow{} Y \xrightarrow{} Z \xrightarrow{} 0$ is a short exact sequence of
$N$-complexes. If any two out of the three are exact, then so is the third.
\end{enumerate}
\end{proposition}

\begin{proof}
A proof of the first statement appears as Proposition~1.5 of~\cite{kapranov} and a proof of the second can be found as Lemma~4.4 of~\cite{gill-hovey-generalized derived cats}.
\end{proof}

\begin{definition}\label{def-homotopic maps}
Two chain maps $f,g : X \xrightarrow {} Y$ of $N$-complexes are called \emph{chain homotopic}, or simply \emph{homotopic} if there exists a collection $\{\, s_n : X_n \xrightarrow{} Y_{n+N-1} \,\}$ such that $g_n - f_n = d^{N-1}s_n + d^{N-2}s_{n-1}d + d^{N-3}s_{n-2}d^2 + \cdots + s_{n-(N-1)}d^{N-1}$ for each $n$. More succinctly, we denote this $$g - f = \sum_{i=0}^{N-1} d^{N-1-i}sd^i.$$ If $f$ and $g$ are homotopic, then we write $f \sim g$. We also call a map $f$ \emph{null homotopic} if $f \sim 0$.
\end{definition}

It is easy to check that $\sim$ is an equivalence relation on Hom sets. Furthermore, one can easily check that if $g_1 \sim g_2$, then $g_1f \sim g_2f$. Similarly, if $f_1 \sim f_2$, then $gf_1 \sim gf_2$. It follows that if $f_1 \sim f_2$ and $g_2 \sim g_2$ then $g_1f_1 \sim g_2f_2$. That is, composition respects chain homotopy. This gives us the following definitions.

\begin{definition}
There is a category $N\text{-}\class{K}(R)$, called the \emph{homotopy category of $N$-complexes}, whose objects are the same as those of $N\text{-Ch}(R)$ and whose Hom sets are the $\sim$ equivalence classes of Hom sets in $N\text{-Ch}(R)$. An isomorphism in $N\text{-}\class{K}(R)$ is called a \emph{chain homotopy equivalence}. These are the maps $f : X \xrightarrow{} Y$  for which there exists a map $g : Y \xrightarrow{} X$ such that $gf$ and $fg$ are chain homotopic to the proper identity maps.
\end{definition}

The above definitions clearly extend standard definitions important to chain complexes ($N=2$). The following proposition illuminates this further.

\begin{proposition}\label{prop-homotopic maps induce equal maps on amplitude homology}
$N\text{-}\class{K}(R)$ is an additive category and the canonical functor $\gamma : N\text{-Ch}(R) \xrightarrow{} N\text{-}\class{K}(R)$ defined by $f \mapsto [f]$ is additive. Moreover, the amplitude homology functors ${}_tH_n : N\text{-Ch}(R) \xrightarrow{} R\text{-Mod}$ factor through $\gamma$.
\end{proposition}

\begin{proof}
First we must show that if $f_1 \sim f_2$ and $g_1 \sim g_2$ then $f_1 + g_1 \sim f_2 + g_2$. But if $f_2 - f_1 = \sum_{i=0}^{N-1} d^{N-1-i}sd^i$ and $g_2 - g_1 = \sum_{i=0}^{N-1} d^{N-1-i}td^i$, then adding we get $$(f_2 - f_1)+(g_2 - g_1) = \sum_{i=0}^{N-1} d^{N-1-i}sd^i + \sum_{i=0}^{N-1} d^{N-1-i}td^i$$ from which we get $(f_2 + g_2)-(f_1 + g_1) = \sum_{i=0}^{N-1} d^{N-1-i}(s+t)d^i$, which proves what we want.

Since composition and addition are well defined on homotopy classes, it now follows that $N\text{-}\class{K}(R)$ inherits the bilinear composition from $N\text{-Ch}(R)$, making $N\text{-}\class{K}(R)$ an additive category (since it also inherits the zero object and biproducts). Now setting $\gamma(f) = [f]$ automatically gives an additive functor. To show that ${}_tH_n$ factors through $\gamma$ it is enough to show that if $f$ is null homotopic, then the induced amplitude homology maps ${}_tH_n(f)$ are all zero. This makes a nice exercise but also can be found in~\cite{kapranov} Proposition~1.11.
\end{proof}

\subsection{Model structures and hovey pairs}

In~\cite{hovey}, Hovey described a one-to-one correspondence between well behaved model category structures on an abelian category $\cat{A}$ and so-called cotorsion pairs in $\cat{A}$. A cotorsion pair is essentially a pair of classes of objects $(\class{F},\class{C})$  which are orthogonal with respect to the functor $\Ext^1_{\cat{A}}(-,-)$. For example, if $R$ is a ring and $\class{A}$ is the class of all $R$-modules while $\class{P}$ is the class of all projective modules and $\class{I}$ is the class of all injective modules, then $(\class{P},\class{A})$ and $(\class{A},\class{I})$ are cotorsion pairs. Furthermore if $\class{F}$ is the class of flat modules and $\class{C}$ is the class of cotorsion modules, then $(\class{F},\class{C})$ is a cotorsion pair. The text~\cite{enochs-jenda-book} is a standard reference on cotorsion pairs.

We will use a version of Hovey's correspondence theorem (from~\cite{hovey}) couched in the language of exact categories. The notion of an exact category was also introduced by Quillen in~\cite{quillen-algebraic K-theory}.  An exact category is a pair $(\class{A},\class{E})$ where $\class{A}$ is an additive category and $\class{E}$ is a class of ``short exact sequences'': That is, triples of objects connected by arrows $A \xrightarrow{i} B \xrightarrow{p} C$ such that $i$ is the kernel of $p$ and $p$ is the cokernel of $i$. A map such as $i$ is necessarily a monomorphism while $p$ an epimorphism. In the language of exact categories $i$ is called an \emph{admissible monomorphism} while $p$ is called an \emph{admissible epimorphism}. The class $\class{E}$ of short exact sequences must satisfy several axioms which are inspired by familiar properties of short exact sequences in any abelian category. As a result many concepts that make sense in abelian categories, such as the extension functor $\Ext$ and cotorsion pairs, still make sense in exact categories. The reader should be able to find any needed facts on exact categories, including cotorsion pairs in exact categories, and model structures on exact categories (exact model structures) nicely summarized in Sections~2 and~3 of~\cite{gillespie-exact model structures}. One can also see B\"uhler's paper~\cite{buhler-exact categories} for a very thorough and readable exposition on exact categories. For easy reference we now state Hovey's theorem which is applied in Section~\ref{section-main theorem} to obtain the desired model structure on $N$-complexes. The definition of \emph{thick} is given in Section~\ref{section-main theorem}.

\begin{theorem}[Hovey's correspondence theorem]\label{them-Hovey's theorem}
Let $(\class{A},\class{E})$ be a (weakly idempotent complete) exact category. Then there is a one-to-one correspondence between exact model structures on $\cat{A}$ and complete cotorsion pairs $(\class{Q},\class{R} \cap \class{W})$ and $(\class{Q} \cap \class{W} , \class{R})$ where $\class{W}$ is a thick subcategory of $\cat{A}$. Given a model structure, $\class{Q}$ is the class of cofibrant objects, $\class{R}$ the class of fibrant objects and $\class{W}$ the class of trivial objects. Conversely, given the cotorsion pairs with $\class{W}$ thick, a cofibration (resp. trivial cofibration) is an admissible monomorphism with a cokernel in $\class{Q}$ (resp. $\class{Q} \cap \class{W}$), and a fibration (resp. trivial fibration) is an admissible epimorphism with a kernel in $\class{R}$ (resp. $\class{R} \cap \class{W}$). The weak equivalences are then the maps $g$ which factor as $g = pi$ where $i$ is a trivial cofibration and $p$ is a trivial fibration.
\end{theorem}

Recently a pair of cotorsion pairs $(\class{Q},\class{R} \cap \class{W})$ and $(\class{Q} \cap \class{W} , \class{R})$ as in the above theorem have been referred to as a \emph{Hovey pair}.

\begin{remark} Hovey's theorem in~\cite{hovey} already allowed for ``proper classes'' of short exact sequences (defined in Section~XII.4 of~\cite{homology}) which in fact give rise to exact categories (by an argument that can be found in Theorem~4.3 of Section~XII.4 in~\cite{homology}). However, exact categories are slightly more general in that they allow for certain full subcategories of abelian categories. (For example, the category of all projective $R$-modules along with the collection of all short exact sequences between these modules forms an exact category. However, this can not be construed as an abelian category along with a proper class.) In any case, one needs to make a choice of language. It could be the language of proper classes of short exact sequences in an abelian category or the language of exact categories. For the current paper either would work, but we choose the second.
\end{remark}

\section{Contractible $N$-complexes}\label{section-contractible $N$-complexes}

Recall that a chain complex is contractible if its identity map is null homotopic. In this case, it is rather immediate that the chain complex is the direct sum of disks on its cycle modules. In this section and the next we derive several results on contractible $N$-complexes. Our first result below is a generalization to $N >2$ the decomposition into a direct sum of $N$-disks. One sees that a complication arises immediately when $N >2$.

\begin{definition}
We call an $N$-complex $C$ \emph{contractible} if its identity map $1_C$ is null homotopic.
\end{definition}

\begin{lemma}\label{lemma-splittings}
Suppose we have a map $g : X \xrightarrow{} Y$ of $R$-modules having a ``splitting'' $s : Y \xrightarrow{} X$ satisfying $gsg = g$. Then $X = \ker{g} \oplus \im{sg}$. Moreover, the pair of maps $(g,s)$ restrict to an isomorphism pair $g : \im{sg} \xrightarrow{} \im{g}$, and $s : \im{g} \xrightarrow{} \im{sg}$.
\end{lemma}

\begin{proof}
This is a variation of an elementary result.  We wish to show (i) $X = \ker{g} + \im{sg}$ and (ii) $\ker{g} \cap \im{sg} = 0$. For (i), let $x \in X$. Then one easily checks that $x - sg(x) \in \ker{g}$ and so $x = [x - sg(x)]+sg(x) \in \ker{g} + \im(sg)$. For (ii), say $z \in \ker{g} \cap \im{sg}$ We write $z = sg(x)$ (some $x \in X$) and suppose $g(z) = 0$. Then $0 = gsg(x) = g(x)$. Therefore $sg(x) = 0$ too. So $z = 0$. This proves that $X = \ker{g} \oplus \im{sg}$. It is clear that $g$ restricts to a map $g : \im{sg} \xrightarrow{} \im{g}$, and $s$ to a map $s : \im{g} \xrightarrow{} \im{sg}$. It is easy to check directly that these are isomorphisms and inverses.
\end{proof}

\begin{theorem}\label{them-contractible N-complexes are direct sums of $N$-disks}
An $N$-complex $C$ is contractible if and only if it is a direct biproduct of $N$-disks $$C = \bigoplus_{n \in \Z} D^N_n(M_n) = \prod_{n \in \Z} D^N_n(M_n)$$ for some set of $R$-modules $\{\, M_n \,\}_{n \in \Z}$. In fact, in this case $M_n = {}_1Z_{n-(N-1)}C$.
\end{theorem}

\begin{proof}
First we note that for some set of $R$-modules $\{\, M_n \,\}_{n \in \Z}$ we indeed have $\bigoplus_{n \in \Z} D^N_n(M_n) = \prod_{n \in \Z} D^N_n(M_n)$ since there are only finitely many terms (summands) in each degree. Now suppose we are given such a complex $\bigoplus_{n \in \Z} D^N_n(M_n)$ which we will denote by $X$. We wish to show $X$ is contractible. To do so, we define the maps
$$s_n : M_{n+N-1} \oplus \cdots \oplus M_{n+1} \oplus M_{n} \xrightarrow{} M_{n+2(N-1)} \oplus \cdots \oplus M_{n+N} \oplus M_{n+N-1}$$ by $s_n(x_{N-1} , \cdots , x_{1} , x_{0}) = (0 , \cdots , 0 , x_{N-1})$. It is easy to see that $\{\, s_n \,\}$ is a homotopy showing $1_X \sim 0$.

Next suppose that $C$ is a contractible complex, so $1_C \sim 0$. We will denote the cycle modules ${}_tZ_{n}C$ of $C$ simply by ${}_tZ_{n}$ for this proof. We immediately have from Proposition~\ref{prop-homotopic maps induce equal maps on amplitude homology} that $C$ is $N$-exact. We will show that $C$ is isomorphic to the direct sum $\bigoplus_{n \in \Z} D^N_n({}_1Z_{n-(N-1)})$. First, by the definition of contractible, there exists a collection $\{\, s_n : C_n \xrightarrow{} C_{n+N-1} \,\}$ such that $1_{X_n} = d^{N-1}s_n + d^{N-2}s_{n-1}d + d^{N-3}s_{n-2}d^2 + \cdots + s_{n-(N-1)}d^{N-1}$ for each $n$. By composing both sides of the equation with $d^{N-1}$ we get that the differential satisfies $d^{N-1}sd^{N-1} = d^{N-1}$. So according to Lemma~\ref{lemma-splittings}, $s$ is a splitting of $d^{N-1} : C_n \xrightarrow{} C_{n-(N-1)}$ and gives a decomposition $C_n = {}_{N-1}Z_{n} \oplus s[{}_1Z_{n-(N-1)}]$ in each degree. Furthermore, restricting the pair $(d^{N-1},s)$ gives us an isomorphism $d^{N-1} : s[{}_1Z_{n-(N-1)}] \xrightarrow{} {}_1Z_{n-(N-1)}$ with inverse $s : {}_1Z_{n-(N-1)} \xrightarrow{} s[{}_1Z_{n-(N-1)}]$. We view $(C,d)$ as shown below:
\begin{center}
\begin{displaymath}
\xymatrix{
   &  \ar@{-->}[d] &  \ar@{-->}[dl] \\
(n+1) & {}_{N-1}Z_{n+1} \ar@{->}[d] \ar@{}[r]|{\bigoplus} & s[{}_1Z_{n+1-(N-1)}] \ar@{->}[dl] \\
(n) & {}_{N-1}Z_{n} \ar@{->}[d] \ar@{}[r]|{\bigoplus} & s[{}_1Z_{n-(N-1)}] \ar@{->}[dl] \\
(n-1)   & {}_{N-1}Z_{n-1} \ar@{}[r]|{\bigoplus} \ar@{-->}[d] & s[{}_1Z_{n-1-(N-1)}] \ar@{-->}[dl] \\
& \\
}
\end{displaymath}
\end{center}
Recall that there is a filtration ${}_1Z_{n} \subseteq {}_{2}Z_{n} \subseteq{} \cdots \subseteq {}_{N-2}Z_{n} \subseteq {}_{N-1}Z_{n}$. The plan now is to continue to show that ${}_{N-2}Z_{n}$ is a direct summand of ${}_{N-1}Z_{n}$ and likewise ${}_{N-3}Z_{n}$ is a direct summand of ${}_{N-2}Z_{n}$ and so on... So we start now by claiming ${}_{N-1}Z_{n} = {}_{N-2}Z_{n} \oplus ds[{}_1Z_{n+1-(N-1)}]$. To prove this we will show (i) ${}_{N-1}Z_{n} = {}_{N-2}Z_{n} + ds[{}_1Z_{n+1-(N-1)}]$ and (ii) ${}_{N-2}Z_{n} \cap ds[{}_1Z_{n+1-(N-1)}] = 0$. For (i), let $z \in {}_{N-1}Z_{n}$. Then by $N$-exactness we know there exists $x \in X_{n+1}$ such that $z = dx$. But we know $x = z' + s(z'')$ for some $z' \in {}_{N-1}Z_{n+1}$ and $z'' \in s[{}_1Z_{n+1-(N-1)}]$. So $z = d(z' + s(z'')) = dz' + ds(z'') \in {}_{N-2}Z_{n} + ds[{}_1Z_{n+1-(N-1)}]$. To show (ii), suppose that $x \in {}_{N-2}Z_{n} \cap ds[{}_1Z_{n+1-(N-1)}]$. Then $d^{N-2}x = 0$ but also $x = ds(z)$ for some $z \in {}_1Z_{n+1-(N-1)}$. So $0 = d^{N-2}x = d^{N-1}s(z)$. But since we know $d^{N-1} : s[{}_1Z_{n+1-(N-1)}] \xrightarrow{} {}_1Z_{n+1-(N-1)}$ is an isomorphism with inverse $s : {}_1Z_{n+1-(N-1)} \xrightarrow{} s[{}_1Z_{n+1-(N-1)}]$ we get $d^{N-1}s(z) = z$. So $0 = z$. Therefore $x = 0$ too. This completes the proof of (ii) and so we have shown ${}_{N-1}Z_{n} = {}_{N-2}Z_{n} \oplus ds[{}_1Z_{n+1-(N-1)}]$. We note that the restricted differential $d : s[{}_1Z_{n+1-(N-1)}] \xrightarrow{} ds[{}_1Z_{n+1-(N-1)}]$ is an isomorphism with inverse $sd^{N-2}$. This is because $(sd^{N-2} \circ d)(s[{}_1Z_{n+1-(N-1)}]) = sd^{N-1}s[{}_1Z_{n+1-(N-1)}] = s[{}_1Z_{n+1-(N-1)}]$, and on the other hand we have $(d \circ sd^{N-2})(ds[{}_1Z_{n+1-(N-1)}]) = dsd^{N-1}s[{}_1Z_{n+1-(N-1)}] = ds[{}_1Z_{n+1-(N-1)}]$. As a result we may now view $(C,d)$ as shown below:
\begin{center}
\begin{displaymath}
\xymatrix{
   &  \ar@{-->}[d] &  \ar@{-->}[dl] &  \ar@2{--}[dl] \\
(n+1) & {}_{N-2}Z_{n+1} \ar@{->}[d] \ar@{}[r]|{\bigoplus}  & ds[{}_1Z_{n+2-(N-1)}] \ar@{->}[dl] \ar@{->}[d] \ar@{}[r]|{\bigoplus} & s[{}_1Z_{n+1-(N-1)}] \ar@{=}[dl] \\
(n) & {}_{N-2}Z_{n} \ar@{->}[d] \ar@{}[r]|{\bigoplus} & ds[{}_1Z_{n+1-(N-1)}] \ar@{->}[dl] \ar@{->}[d] \ar@{}[r]|{\bigoplus} & s[{}_1Z_{n-(N-1)}] \ar@{=}[dl] \\
(n-1)  & {}_{N-2}Z_{n-1} \ar@{-->}[d] \ar@{}[r]|{\bigoplus}  & ds[{}_1Z_{n-(N-1)}] \ar@{}[r]|{\bigoplus} \ar@{-->}[d] & s[{}_1Z_{n-1-(N-1)}] \ar@2{--}[dl] \\
& & \\
}
\end{displaymath}
\end{center}

A similar argument shows  ${}_{N-2}Z_{n} = {}_{N-3}Z_{n} \oplus d^2s[{}_1Z_{n+2-(N-1)}]$ and so we get:

\begin{center}
\begin{displaymath}
\xymatrix{
 \ar@{-->}[d] &  \ar@{-->}[dl] &  \ar@2{--}[dl] &  \ar@2{--}[dl] \\
{}_{N-3}Z_{n+1} \ar@{->}[d] \ar@{}[r]|{\bigoplus} & d^2s[{}_1Z_{n+3-(N-1)}] \ar@{->}[dl] \ar@{}[r]|{\bigoplus} & ds[{}_1Z_{n+2-(N-1)}] \ar@{=}[dl] \ar@{}[r]|{\bigoplus} & s[{}_1Z_{n+1-(N-1)}] \ar@{=}[dl] \\
{}_{N-3}Z_{n} \ar@{->}[d] \ar@{}[r]|{\bigoplus} & d^2s[{}_1Z_{n+2-(N-1)}] \ar@{->}[dl] \ar@{}[r]|{\bigoplus} & ds[{}_1Z_{n+1-(N-1)}] \ar@{=}[dl] \ar@{}[r]|{\bigoplus} & s[{}_1Z_{n-(N-1)}] \ar@{=}[dl] \\
{}_{N-3}Z_{n-1} \ar@{-->}[d] \ar@{}[r]|{\bigoplus}  & d^2s[{}_1Z_{n+1-(N-1)}] \ar@{-->}[dl] \ar@{}[r]|{\bigoplus} & ds[{}_1Z_{n-(N-1)}] \ar@2{--}[dl] \ar@{}[r]|{\bigoplus} & s[{}_1Z_{n-1-(N-1)}] \ar@2{--}[dl] \\
& & & \\
}
\end{displaymath}
\end{center}
Continuing in this way we are led to a decomposition $C =  \bigoplus_{n \in \Z} D^N_n({}_1Z_{n-(N-1)})$.

\end{proof}

\begin{proposition}\label{prop-characterizing maps in and out of contractible complexes}
Let $C$ be contractible. So we may assume $C = \bigoplus_{n \in \Z} D^N_n(M_n)$.
\begin{enumerate}
\item Any collection of maps $\{\, u_n : X_n \xrightarrow{} M_{n+N-1} \,\}$ determines a chain map $\beta : X \xrightarrow{} C$ by setting $\beta_n = (u_n , u_{n-1} d_X , u_{n-2} d^2_X , \cdots , u_{n-(N-1)}d^{N-1}_X)$. Conversely, any chain map $\beta : X \xrightarrow{} C$ is equivalent to a collection of maps $\{\, u_n : X_n \xrightarrow{} M_{n+N-1} \,\}$ satisfying this condition.
\item Any collection of maps $\{\, q_n : M_n \xrightarrow{} Y_{n} \,\}$ determines a chain map $p : C \xrightarrow{} Y$ by setting $p_n = d_Y^{N-1}q_{n+(N-1)} + \cdots + d_Y^2q_{n+2} + d_Yq_{n+1} + q_n$. Conversely, any chain map $p : C \xrightarrow{} Y$ is equivalent to a collection of maps $\{\, q_n : M_n \xrightarrow{} Y_{n} \,\}$ satisfying this condition.
\end{enumerate}
\end{proposition}

\begin{proof}
Assume we have a collection of maps $\{\, u_n : X_n \xrightarrow{} M_{n+N-1} \,\}$ as in (1). Then it is easy to check that the diagram below commutes and so $\beta = \{\beta_n\}$ as defined is a chain map.
$$\begin{CD}
  X_n @>\beta_n >> M_{n+N-1} \oplus \cdots \oplus M_{n+1} \oplus M_{n} \\
  @Vd_XVV         @VVV                         \\
  X_{n-1} @>\beta{n-1} >> M_{n+N-2} \oplus \cdots \oplus M_{n} \oplus M_{n-1} \\
\end{CD}$$
On the other hand, suppose $\beta : X \xrightarrow{} C$ is any chain map. Then for each $n$ we must have $\beta_n = (u_n , u'_n , u''_n , \cdots , u^{N-1}_n)$ for some maps $u_n , u'_n , u''_n , \cdots , u^{N-1}_n$. One can check that commutativity of the above diagram leads to the following relations:
$$u'_n = u_{n-1} d_X  \ , \ \ u''_n = u'_{n-1} d_X  \ , \cdots , \ \ u^{N-1}_n = u^{N-2}_{n-1} d_X.$$ Then solving for each of these in terms of the $u_i$'s we get $$\beta_n = (u_n , u'_n , u''_n , \cdots , u^{N-1}_n) = (u_n , u_{n-1} d_X , u_{n-2} d^2_X , \cdots , u_{n-(N-1)}d^{N-1}_X).$$

The proof of (2) can be checked in a similar way.
\end{proof}

\begin{corollary}\label{cor-factoring through a contractible}
Let $f,g : X \xrightarrow{} Y$ be chain maps of $N$-complexes. Then $f \sim g$ if and only if $g - f$ factors through a contractible complex.
\end{corollary}

\begin{proof}
It is enough to show $f$ is null homotopic if and only if $f$ factors through a contractible. So assume $f \sim 0$. Then there exists a collection of maps $\{\, s_n : X_n \xrightarrow{} Y_{n+N-1} \,\}$ such that $f_n = d^{N-1}s_n + d^{N-2}s_{n-1}d + d^{N-3}s_{n-2}d^2 + \cdots + s_{n-(N-1)}d^{N-1}$ for each $n$. By part (1) of Proposition~\ref{prop-characterizing maps in and out of contractible complexes}, the collection $\{\, s_n : X_n \xrightarrow{} Y_{n+N-1} \,\}$ determines a chain map $\beta : X \xrightarrow{} \bigoplus_{n \in \Z} D^N_n(Y_n)$ where $$\beta_n = (s_n , s_{n-1} d_X , s_{n-2} d^2_X , \cdots , s_{n-(N-1)}d^{N-1}_X).$$ Furthermore, by part (2) of Proposition~\ref{prop-characterizing maps in and out of contractible complexes}, the identity maps $\{\, 1_{Y_n} : Y_n \xrightarrow{} Y_{n} \,\}$ determine a chain map $p : \bigoplus_{n \in \Z} D^N_n(Y_n) \xrightarrow{} Y$ where $$p_n = d_Y^{N-1} + \cdots + d_Y^2 + d_Y + 1_{Y_n}.$$ This shows that $f$ factors through the contractible complex $\bigoplus_{n \in \Z} D^N_n(Y_n)$ since $$p_n \beta_n = d^{N-1}s_n + d^{N-2}s_{n-1}d + d^{N-3}s_{n-2}d^2 + \cdots + s_{n-(N-1)}d^{N-1} = f_n.$$

On the other hand, suppose $f$ factors through some contractible complex $C = \bigoplus_{n \in \Z} D^N_n(M_n)$. So $f = p \beta$ where $\beta : X \xrightarrow{} C$ and $p : C \xrightarrow{} Y$. Then by Proposition~\ref{prop-characterizing maps in and out of contractible complexes} we get $\beta_n = (u_n , u_{n-1} d_X , u_{n-2} d_X^2 , \cdots , u_{n-(N-1)}d_X^{N-1})$ for some collection $\{\, u_n : X_n \xrightarrow{} M_{n+N-1} \,\}$ and $p : C \xrightarrow{} Y$ must take the form $p_n = d_Y^{N-1}q_{n+(N-1)} + \cdots + d_Y^2q_{n+2} + d_Yq_{n+1} + q_n$ where $\{\, q_n : M_n \xrightarrow{} Y_n \,\}$ is some collection of maps. Composing we get $p_n \beta_n =$ $$d^{N-1}q_{n+(N-1)}u_n + d^{N-2}q_{n+(N-2)} u_{n-1} d + \cdots + dq_{n+1} u_{n-(N-2)}d^{N-2} + q_n u_{n-(N-1)}d^{N-1}.$$ Now setting $s_n = q_{n+(N-1)}u_n$ we get a collection of maps $\{\, s_n : X_n \xrightarrow{} Y_{n+N-1} \,\}$ satisfying $f_n = d^{N-1}s_n + d^{N-2}s_{n-1}d + d^{N-3}s_{n-2}d^2 + \cdots + s_{n-(N-1)}d^{N-1}$. By definition, we get $f \sim 0$.
\end{proof}

\begin{corollary}\label{cor-contractible complexes closed under retracts, etc.}
The class of contractible complexes is closed under direct sums, products and retracts (direct summands).
\end{corollary}

\begin{proof}
First note that for a fixed $n$, we have $\bigoplus_{i \in I} D^N_n(M_i) = D^N_n(\bigoplus_{n \in \Z} M_i)$. Using this observation, given a direct sum $\bigoplus_{i \in I} C_i$ of contractible complexes, it will again be contractible by applying Theorem~\ref{them-contractible N-complexes are direct sums of $N$-disks} and reshuffling the summands. A similar argument with products applies to show that a product of contractible complexes is again contractible.

We now show that a retract (direct summand) of a contractible complex is again contractible. So suppose $C$ is contractible and suppose $i : S \xrightarrow{} C$ and $r : C \xrightarrow{} S$ are chain maps with $ri = 1_S$. Then by Corollary~\ref{cor-factoring through a contractible} we conclude that $1_S \sim 0$, which means $C$ is contractible.
\end{proof}

\section{Main Theorem}\label{section-main theorem}

We now use the results of the previous section along with Hovey's correspondence Theorem~\ref{them-Hovey's theorem} to show there is a model structure on the category of $N$-complexes whose homotopy category recovers $N\text{-}\class{K}(R)$. We use the language of exact model structures from~\cite{gillespie-exact model structures}.

Let $N\text{-Ch}(R)_{dw}$ be the exact category $(\cat{A},\class{E})$, where $\cat{A}$ is the category $N\text{-Ch}(R)$ and $\class{E}$ is the class of all degreewise split short exact sequences of $N$-complexes. Then one can check that $N\text{-Ch}(R)_{dw}$ is a weakly idempotent complete exact category. [Checking this is rather trivial and we refer the reader to Section~2 of~\cite{gillespie-exact model structures} for the checklist of properties. But most of this is immediate: The most nontrivial thing required here is that pushouts (and pullbacks) of $N$-complexes are taken degreewise and that any pushout (or pullback) of a split exact sequence of $R$-modules is still split exact.]

\begin{proposition}\label{prop-contractible complexes are projective and injective}
The following statements are equivalent for an $N$-complex $C$.
\begin{enumerate}
\item $C$ is contractible.
\item $C$ is a projective object in $N\text{-Ch}(R)_{dw}$.
\item $C$ is an injective object in $N\text{-Ch}(R)_{dw}$.
\end{enumerate}
\end{proposition}

\begin{proof}
We will show $C$ is contractible if and only if it is projective in $N\text{-Ch}(R)_{dw}$. The proof for injectives ought to be similar.

First it is easy to check that a disk $D^N_n(M)$ on any module $M$ is projective in $N\text{-Ch}(R)_{dw}$. Indeed by Proposition~11.3 of~\cite{buhler-exact categories}, all that is required is to show that any degreewise split epimorphism $Y \xrightarrow{} D^N_n(M)$ splits. Given such an epimorphism means there is an $N$-complex $X$ and a degreewise split short exact sequence $0 \xrightarrow{} X \xrightarrow{} Y \xrightarrow{} D^N_n(M) \xrightarrow{} 0$ of $N$-complexes. Since degreewise split, we have $Y_k = X_k \oplus M$ for $k = n, n-1, \cdots, n - (N-1)$ and $Y_k = X_k$ for all other $k$ and the epimorphism $Y \xrightarrow{} D^N_n(M)$ takes the form $X_k \oplus M \xrightarrow{\pi} M$ where $\pi$ is the canonical projection. One can check that the differential of $Y$ is completely determined by a collection of maps $\{\, s_1 , s_2 , \cdots , s_N  \,\}$ with the $s_i : M \xrightarrow{} X_{n-i}$ collectively satisfying a condition. Regardless of this condition, we use the maps $s_i$ to define a splitting $D^N_n(M) \xrightarrow{s} Y$ induced by defining it in degree $n$ to be $(0,1_M) : M \xrightarrow{} X_n \oplus M$. Then in degrees $n-i$ (for $i = 1,2, \cdots, N-1$) the splitting takes the form $$(d^{i-1}s_1 + d^{i-2}s_2 + \cdots + d s_{i-1} + s_i , 1_M) : M \xrightarrow{} X_{n-i} \oplus M.$$

Next suppose $C$ is contractible and write  $C = \bigoplus_{n \in \Z} D^N_n(M_n)$ using Theorem~\ref{them-contractible N-complexes are direct sums of $N$-disks}. Then since each $D^N_n(M_n)$ is projective in $N\text{-Ch}(R)_{dw}$ and since $C$ is a direct sum of projectives it follows from Corollary~11.7 of~\cite{buhler-exact categories} that $C$ is projective in the exact category $N\text{-Ch}(R)_{dw}$.
\end{proof}

Recall that by a \emph{thick} subcategory we mean a class of objects $\class{W}$ which is closed under direct summands and satisfies the property that if two out of three terms in a short exact sequence are in $\class{W}$, then so is the third.

\begin{proposition}\label{prop-enough contractibles and thickness}
Let $\class{W}$ be the class of contractible $N$-complexes.
\begin{enumerate}
\item $\class{W}$ is a thick subcategory of $N\text{-Ch}(R)_{dw}$.
\item $N\text{-Ch}(R)_{dw}$ has enough projectives and enough injectives. That is, given an $N$-complex $X$, there exists $C,D \in \class{W}$, a degreewise split epimorphism $C \xrightarrow{} X$ (enough projectives) and a degreewise split monomorphism $X \xrightarrow{} D$ (enough injectives).
\end{enumerate}
\end{proposition}

\begin{proof}
First, by Corollary~\ref{cor-contractible complexes closed under retracts, etc.} we know that $\class{W}$ is closed under taking direct summands. Next suppose that $0 \xrightarrow{} X \xrightarrow{} Y \xrightarrow{} Z \xrightarrow{} 0$ is a degreewise split short exact sequence of $N$-complexes. If $Z$ is in $\class{W}$ then the sequence splits by Proposition~\ref{prop-contractible complexes are projective and injective}, making $X$ a direct summand of $Y$. So if $Y$ is in $\class{W}$, then $X$ must also be in $\class{W}$ by Corollary~\ref{cor-contractible complexes closed under retracts, etc.}.  This proves that $Y, Z$ being in $\class{W}$ implies $X$ is in $\class{W}$. The dual argument holds and shows  $X, Y \in \class{W}$ implies $Z \in \class{W}$. Finally suppose $X$ and $Z$ are in $\class{W}$. Then by Proposition~\ref{prop-contractible complexes are projective and injective} it is clear that $Y = X \oplus Z$. So $Y \in \class{W}$ by Corollary~\ref{cor-contractible complexes closed under retracts, etc.}. This proves $\class{W}$ is thick.

We prove only the enough projectives portion of the second statement. For this let $p : \bigoplus_{n \in \Z} D^N_n(X_n) \xrightarrow{} X$ be induced from the set of identity maps $\{\, 1_{X_n} : X_n \xrightarrow{} X_{n} \,\}$. Then note that in degree $n$ we have
$$p_n : X_{n+N-1} \oplus \cdots \oplus X_{n+1} \oplus X_{n} \xrightarrow{d^{N-1} + \cdots + d + 1} X_n$$ which is clearly an epimorphism. Now define an $N$-complex $K$ by setting  $K_n = X_{n+N-1} \oplus \cdots \oplus X_{n+2} \oplus X_{n+1}$ and with differential defined by $$d(x_{N-1} , \cdots , x_{2} , x_{1}) = (x_{N-2} , \cdots , x_2 , x_{1}, -d^{N-1}x_{N-1} - \cdots - d^2x_{2} - dx_{1}).$$ One can check this differential makes $K$ an $N$-complex. Now we have a chain map $i : K \xrightarrow{} \bigoplus_{n \in \Z} D^N_n(X_n)$ defined in each degree via $$i_n = (1,1, \cdots , 1, -d^{N-1} - \cdots - d^2 - d).$$ It is easy to check that $$0 \xrightarrow{} K \xrightarrow{i} \bigoplus_{n \in \Z} D^N_n(X_n) \xrightarrow{p} X \xrightarrow{} 0$$ is a degreewise split short exact sequence.
\end{proof}

\begin{remark}
We didn't actually need to describe the complex $K$ in the proof of Proposition~\ref{prop-enough contractibles and thickness}. But we do so to point out now that it can be taken to serve as the loop on $X$. That is, $\Omega X$. The dual construction produces the suspension $\Sigma X$.
\end{remark}

\begin{theorem}\label{them-the homotopy category of $N$-complexes is a homotopy category}
Let $\class{A}$ denote the class of all $N$-complexes and let $\class{W}$ denote the class of all contractible complexes. Both $(\class{A},\class{W})$ and $(\class{W},\class{A})$ are complete cotorsion pairs in  $N\text{-Ch}(R)_{dw}$, and so form a Hovey pair. The corresponding model structure on $\text{Ch}(R)_{dw}$ is described as follows. The cofibrations (resp. trivial cofibrations) are the degreewise split monomorphisms (resp. split monomorphisms with contractible cokernel) and the fibrations (resp. trivial fibrations) are the degreewise split epimorphisms (resp. split epimorphisms with contractible kernel). The weak equivalences are the homotopy equivalences. We note the following properties of this model structure:
\begin{enumerate}
\item The model structure is Frobenius. In particular, each $N$-complex is both cofibrant and fibrant.
\item The formal homotopy relation coincides with the notion of chain homotopy in Definition~\ref{def-homotopic maps} and two maps are chain homotopic if and only if their difference factors through a contractible complex.
\item $\text{HoCh}(R)_{dw} = N\text{-}\class{K}(R)$.
\end{enumerate}
\end{theorem}

\begin{proof}
It follows immediately from Proposition~\ref{prop-contractible complexes are projective and injective}~(2) that $(\class{W},\class{A})$ is a cotorsion pair in $N\text{-Ch}(R)_{dw}$ and Proposition~\ref{prop-contractible complexes are projective and injective}~(3) says that $(\class{A},\class{W})$ is a cotorsion pair. Proposition~\ref{prop-enough contractibles and thickness}~(2) says that these cotorsion pairs are complete. Also by Proposition~\ref{prop-enough contractibles and thickness}~(1), $\class{W}$ is thick and so $(\class{A},\class{W})$ and $(\class{W},\class{A})$ form a Hovey pair where in Theorem~\ref{them-Hovey's theorem} we have $\class{A} = \class{Q} = \class{R}$ and $\class{W}$ are the trivial objects. The existence of the  model structure follows and as in~\cite{gillespie-exact model structures} we call it Frobenius since it exists on an exact category and each object is both cofibrant and fibrant. 

It was shown in Corollary~4.8~(3) of~\cite{gillespie-exact model structures} that for any Frobenius model structure, two maps are homotopic if and only if their difference factors through a projective-injective object. So the second statement now follows from Corollary~\ref{cor-factoring through a contractible} and Proposition~\ref{prop-contractible complexes are projective and injective}. The third statement is clear from the most fundamental theorem about model categories: See Theorem~1.2.10 of~\cite{hovey-model-categories}. 

\end{proof}


\end{document}